\documentclass[12pt,reqno]{amsart}
\usepackage{amssymb}
\usepackage{graphicx}
\usepackage{xcolor}
\usepackage{bm}

\usepackage[all]{xy}
\DeclareFontFamily{U}{mathb}{\hyphenchar\font45}
\DeclareFontShape{U}{mathb}{m}{n}{
      <5> <6> <7> <8> <9> <10> gen * mathb
      <10.95> mathb10 <12> <14.4> <17.28> <20.74> <24.88> mathb12
      }{}
\DeclareSymbolFont{mathb}{U}{mathb}{m}{n}
\DeclareMathSymbol{\righttoleftarrow}{3}{mathb}{"FD}
\DeclareFontFamily{U}{wncy}{}
\DeclareFontShape{U}{wncy}{m}{n}{<->wncyr10}{}
\DeclareSymbolFont{mcy}{U}{wncy}{m}{n}
\DeclareMathSymbol{\Sh}{\mathord}{mcy}{"58}

\theoremstyle{plain}
\newtheorem{prop}{Proposition}[section]
\newtheorem{theo}[prop]{Theorem}
\newtheorem{coro}[prop]{Corollary}
\newtheorem{lemm}[prop]{Lemma}
\theoremstyle{remark}
\newtheorem{rema}[prop]{Remark}

\theoremstyle{definition}

\newtheorem{exam}[prop]{Example}

\numberwithin{equation}{section}

\newcommand{\bP}{{\mathbb P}}
\newcommand{\Q}{{\mathbb Q}}

\newcommand{\G}{{\mathbb G}}

\newcommand{\Z}{{\mathbb Z}}

\newcommand{\rH}{{\mathrm H}}

\newcommand{\GL}{{\mathrm{GL}}}

\newcommand{\ra}{\rightarrow}
\newcommand{\lra}{\longrightarrow}

\newcommand{\bQ}{{\mathbb Q}}

\newcommand{\bZ}{{\mathbb Z}}

\newcommand{\eqto}{\stackrel{\lower1.5pt\hbox{$\scriptstyle\sim\,$}}\to}
\newcommand{\eqdashto}{\stackrel{\lower1.5pt\hbox{$\scriptstyle\sim\,$}}\dashrightarrow}
\newcommand{\actsfromleft}{\mathrel{\reflectbox{$\righttoleftarrow$}}}
\newcommand{\actsfromright}{\righttoleftarrow}

\DeclareMathOperator{\Gal}{Gal}

\DeclareMathOperator{\Pic}{Pic}
\DeclareMathOperator{\Spec}{Spec}

\DeclareMathOperator{\Hom}{Hom}

\DeclareMathOperator{\Br}{Br}
\DeclareMathOperator{\Aut}{Aut}

\DeclareMathOperator{\Burn}{Burn}

\begin{document}
\title[Plane Cremona group]{Cohomology of finite subgroups of the plane Cremona group}

\author{Andrew Kresch}
\address{
  Institut f\"ur Mathematik,
  Universit\"at Z\"urich,
  Winterthurerstrasse 190,
  CH-8057 Z\"urich, Switzerland
}
\email{andrew.kresch@math.uzh.ch}
\author{Yuri Tschinkel}
\address{
  Courant Institute,
  251 Mercer Street,
  New York, NY 10012, USA
}

\email{tschinkel@cims.nyu.edu}

\address{Simons Foundation\\
160 Fifth Avenue\\
New York, NY 10010\\
USA}

\date{March 3, 2022}

\begin{abstract}
An equivariant stable birational invariant of an action of a finite group on a smooth projective 
variety is the first 
cohomology group of the Picard module.
Bogomolov--Prokhorov and Shinder computed this for actions of cyclic groups on rational surfaces, with maximal stabilizers, 
in terms of the geometry of the fixed point locus. 
Using the Brauer group of the quotient stack, we extend the computation to more general actions and  
relate the computation to the equivariant Burnside group formalism.
\end{abstract}

\maketitle

\section{Introduction}
\label{sec.intro}
Consider a smooth projective variety $X$ over a field $k$ of characteristic zero, with a fixed algebraic closure $\bar{k}/k$. Let $G$ be a profinite group. Following Manin~\cite{manin2}, we say that $X$ is a $G$-variety if we have an action of $G$ on $X_{\bar{k}}$. Manin distinguished two cases:
\begin{itemize}
\item {\em Algebraic:} $G=\Gal(\bar{k}/k)$ and the action is via
the action on $\bar{k}$,
\item {\em Geometric:} $k=\bar{k}$ and $G$ is finite, 
acting by $G\hookrightarrow \Aut(X)$.
\end{itemize}
In the algebraic case, we are interested in (stable) \emph{$k$-rationality} of $X$.
In the geometric case, we are interested in (stable) \emph{linearizability}, where linearizability means the $G$-equivariant birationality of $X$ and $\bP(V)$
for some representation $G\to \GL(V)$,
and stable linearlizability means the same for $X\times \bP^m$, for some $m$ (where $G$ acts trivially on $\bP^m$).

One of the insights in \cite{manin} and \cite{manin2} was that there are striking similarities between the study of $k$-rationality problems for geometrically rational surfaces $X$ over
nonclosed fields $k$ and the study of $G$-surfaces, up to $G$-equivariant birationality, over algebraically closed fields.  

Of fundamental importance, in both cases, is the cohomology group
\begin{equation}
\label{eqn:pic}
\rH^1(H, \Pic(X_{\bar{k}})),
\end{equation}
where $H\subseteq G$ is a finite-index subgroup. Its vanishing, for all $H$, is a necessary condition for stable $k$-rationality, in the algebraic case, respectively, stable linearizability, in the geometric case. In principle, \eqref{eqn:pic} is computable, provided one can reconstruct the $H$-action on $\Pic(X_{\bar{k}})$. In practice, in the algebraic case, this amounts to the determination of the Galois action on exceptional curves, or the Hasse-Weil $L$-function -- both computationally intensive problems. 
This has been implemented for del Pezzo surfaces of degree 4 \cite{bright}, and diagonal del Pezzo surfaces of degree $\le 3$ in \cite{ct-kanevsky}, \cite{KT-dp2}, \cite{tony}. 

On the other hand, 
Manin suggested (see, e.g., \cite[Thm.\ 5.9]{manin2}) to use information about fixed-point loci for the action of $G$ -- in the geometric case, this is easily available from the equations of the surface.   

From now on, we focus on the geometric case; we will assume that $k$ is algebraically closed and $X$ is a rational surface. The starting point for our paper was a theorem of Bogomolov and Prokhorov \cite{BP}: a 
cyclic group $G=C_p$ of prime order $p$, acting regularly on $X$, can have at most one irreducible (smooth) curve $C$ of genus $g\ge 1$ in its fixed locus, and the cohomology is given by 
$$
\rH^1(G, \Pic(X)) \cong (\bZ/p\bZ)^{2\mathsf{g}}, 
$$
where $\mathsf{g}=\mathsf{g}(C)$ is the genus of $C$. This was extended by Shinder \cite{shinder} to arbitrary cyclic actions, under the assumption that all stabilizers are maximal. 

Our main result is an algorithm that allows a direct computation of $\rH^1(G,\Pic(X))$ in the presence of fixed points, in particular, for cyclic $G$,  
in terms of curves with nontrivial generic stabilizer. This is very close, in spirit, to the Burnside group formalism from \cite{BnG}, providing new invariants in equivariant birational geometry. 

We approach the computation of $\rH^1(G,\Pic(X))$ 
via the Brauer group $\Br([X/G])$ of the {\em quotient stack} $[X/G]$, or
more classically, the {\em equivariant} Brauer group, introduced in \cite{FW}, see also \cite[Sect.\ 2.3]{HT-quad}.

In Section \ref{sect:background}, we recall basic notions concerning group cohomology and the plane Cremona group.
In Section~\ref{sect:stable}, we discuss stable birational invariants of $G$-actions, related via the exact sequence 
$$
\Pic(X)^G\stackrel{\delta_2}{\lra}\rH^2(G,k^\times) \to \Br([X/G]) \to \rH^1(G,\Pic(X)) \stackrel{\delta_3}{\lra} \rH^3(G, k^\times).
$$
This sequence 
determines $\rH^1(G,\Pic(X))$ in the presence of fixed points, by the triviality of $\delta_2$ and $\delta_3$, in this case.
We compute  $\Br([X/G])$ in Section~\ref{sect:brauer}, from information on curves with nontrivial generic stabilizer. We include representative examples of such computations in Section~\ref{sect:exam} and comment on connections with equivariant Burnside groups in Section~\ref{scn.comparison}.

\medskip
\noindent
{\bf Acknowledgments:} 
We are grateful to Yuri Prokhorov for his interest and helpful comments. The first author was partially supported by the Swiss National Science Foundation. The second author was partially supported by NSF grant 2000099.

\section{Generalities}
\label{sect:background}
We work over an algebraically closed field $k$ of characteristic zero.

\subsection{Cohomology of finite groups}
Throughout, $G$ is a finite group, and
$\rH^i(G,M)$
the group cohomology with coefficients in a $G$-module $M$. We write $M^G$ for the submodule of $G$-invariant elements in $M$. 

We are interested in $M=k^\times$, with trivial action. 
The $\Q$-module structure on the quotient $M/M_{\mathrm{tors}}$ by the torsion subgroup $M_{\mathrm{tors}}=\mu_{\infty}$ (roots of unity) gives rise to an identification 
$$
\rH^i(G,k^\times)\cong \rH^i(G,\mu_{\infty})\quad \text{for }  i>0,
$$
and in particular, the independence of $k$.
We record computations for several particular groups.
For a cyclic group $C_m=\Z/m\Z$, we have
\[ \rH^2(C_m,k^\times)=0,\quad
\rH^3(C_m,k^\times)\cong\Z/m. \]
For a bicyclic group 
$G:=\Z/m\Z\oplus \Z/n\Z$, with $d:=\gcd(m,n)$,
\[ \rH^2(G,k^\times)\cong\Z/d,
\quad \rH^3(G,k^\times)\cong\Z/m\oplus \Z/d\oplus \Z/n. \]
For a tricyclic group $G:=\Z/m_1\Z\oplus \Z/m_2\Z\oplus \Z/m_3\Z$, we put
\[ d_{ij}:=\gcd(m_i,m_j), \quad
d:=\gcd(m_1,m_2,m_3), \]
and then
\[
\rH^2(G,k^\times)\cong\bigoplus_{i<j}\Z/d_{ij},
\quad
\rH^3(G,k^\times)\cong\bigoplus_i \Z/m_i\oplus\bigoplus_{i<j}\Z/d_{ij}\oplus\Z/d.
\]
For the dihedral group $D_8$ of order $8$, we have
$$
\rH^2(D_8,k^\times) \cong \bZ/2, \quad 
\rH^3(D_8,k^\times) \cong (\bZ/2)^2\oplus \bZ/4.
$$

Recall that, for any (pro)finite $G$, we have
\begin{equation}
\label{eqn:h1-2}
\rH^1(G,\bQ/\bZ) \cong \rH^2(G,\bZ),
\end{equation}
coming from the exact sequence 
$$
0\ra \bZ\to \bQ\to \bQ/\bZ\to 0.
$$
We also have
\[
\rH^i(G,M) \hookrightarrow \bigoplus_{\ell} \, \rH^i(\mathrm{Syl}_{\ell}, M)^{N_G(\mathrm{Syl}_{\ell})},
\]
where $\mathrm{Syl}_{\ell}=\mathrm{Syl}_{\ell}(G)$ is an $\ell$-Sylow subgroup of $G$ and $N_G(\mathrm{Syl}_{\ell})$ is its normalizer. This allows to separate $\ell$-primary components of $\rH^i$. 

\subsection{The Cremona group and its finite subgroups}
\label{sect:cre}
The plane Cremona group $\mathrm{Cr}_2$ is the group of birational automorphisms of $\bP^2$. Each finite subgroup $G\subset \mathrm{Cr}_2$ can be realized as a subgroup of (regular) automorphisms of a smooth projective rational surface. The classification of groups that can occur, and the determination of the conjugacy classes of these groups in $\mathrm{Cr}_2$, has occupied generations of mathematicians, culminating in \cite{blanc-thesis}, for abelian $G$, and \cite{DI}, in general. 

The standard approach to these problems relies on the equivariant Minimal Model Program: It suffices to consider {\em minimal} $G$-surfaces, which in turn are $G$-isomorphic to either conic bundles over $\bP^1$, or del Pezzo surfaces. 
We have  $\Pic(X)^G\simeq \bZ^2$ in the first case, respectively, $\simeq \bZ$, in the second case. 
When the anticanonical degree is $\ge 5$, there is no cohomology.

In case of del Pezzo surfaces of degree $\le 4$, the group $G$ 
embeds into the Weyl group of the corresponding
root system on the (primitive) Picard group, i.e., $\mathsf D_5$, $\mathsf E_6$, $\mathsf E_7$, $\mathsf E_8$. 
Finite subgroups of these can be enumerated. Not all of these groups arise as automorphisms of rational surfaces, but for those that do,  equations of these surfaces are written down explicitly.  
In the conic bundle case, the analysis is based on a decomposition of the $G$-action into an action of a quotient group on the base, and an action of the kernel on the generic fiber; see, e.g., \cite{DI}, \cite{pro-2}.

\subsection{Abelian subgroups of the plane Cremona group}
\label{sect:ab}
A classification of finite abelian subgroups $G\subset \mathrm{Cr}_2$ can be found in \cite{blanc-thesis}. We summarize the main features. 

\begin{itemize}
\item The only $G$ that occur are
\cite[Thm.\ 6]{blanc-thesis}:
$$
(\bZ/2)^4,\, (\bZ/3)^2,\, (\bZ/4)^2 \oplus (\bZ/2),\, \bZ/2n \oplus (\bZ/2)^2,\,
\bZ/n\oplus \bZ/m,  \, \forall n,m.
$$
\item If $G$ fixes some curve of positive genus then $G$ is cyclic, of order $\le 6$. If the curve has genus $> 1$, the order is 2 or 3 \cite[Thm.\ 3]{blanc-thesis}.
\item If no (nontrivial) element of $G$ fixes a curve of positive genus, then $G$ is conjugate to
a subgroup of $\Aut(\bP^2)$, $\Aut(\bP^1\times \bP^1)$, or $G=\bZ/2\times \bZ/4$, with a particular action \cite[Thm.\ 5]{blanc-thesis}. 
\item If $G$ is cyclic and no element of $G$ fixes a curve of positive genus then $G$ is conjugate to a subgroup of 
$\Aut(\bP^2)$ \cite[Thm.\ 4]{blanc-thesis}.
\end{itemize}

Cyclic subgroups play a special role in the classification: they always admit fixed points. There are two series of such groups:
\begin{itemize}
\item[$(\mathbf{dJ})$] de Jonqui\`eres, i.e., preserving a pencil of rational curves,
\item[$(\mathbf{ndJ})$] not de Jonqui\`eres. 
\end{itemize}
Nonlinearizable groups of type $(\mathbf{dJ})$ have even order $2n$, with element of order $2$ fixing a hyperelliptic curve $C$ and residual action of $C_n$ on $C$. There are 29 families of conjugacy classes of $(\mathbf{ndJ})$, listed in \cite[Table 2]{blancsubgroups}.

\subsection{Cohomology of finite subgroups of the Cremona group}
\label{sect:bog}
In the del Pezzo case, all possible cohomology groups
$$
\rH^1(G,\Pic(X))
$$
have been computed, see, e.g., \cite{kst}, \cite{swd}, \cite[Thm. 4.1]{corn}, 
using the presentation of $G$ as a 
subgroup of the relevant Weyl group, as explained in Section~\ref{sect:cre}. However, not all subgroups of the respective Weyl groups are realizable as automorphisms. 
In the conic bundle case, this computation can be found essentially in \cite[Sect.\ 3A]{CTSSD}: the cohomology group is always $2$-torsion.

A theorem of Bogomolov and Prokhorov \cite{BP} states:
$$
\rH^1(G, \Pic(X)) \cong (\bZ/p\bZ)^{2\mathsf{g}}, 
$$
when $G=C_p$ is cyclic of prime order $p$ and 
$\mathsf{g}$ is the genus of the unique curve of positive genus fixed by $G$ (if such a curve exists). It applies to $(\mathbf{dJ})$ and $p=2$, and to five actions of type $(\mathbf{ndJ})$, with $p=2$, $3$, $5$. The proof relies on  classification.  

Shinder \cite{shinder} generalized this result to arbitrary finite cyclic $G=C_n$ acting on surfaces $X$ with 
$\rH_1(X,\bZ)=0$,
under the assumption that \emph{all stabilizers are maximal}. For rational $X$, this applies to two actions, not covered by \cite{BP}, namely,
$\# 5$ ($n=4$) and $\# 17$ ($n=6$) in \cite[Table 2]{blancsubgroups}. 
In Section~\ref{sect:brauer} we provide a general result describing cohomology in terms of fixed loci, with a treatment of arbitrary cyclic actions on rational surfaces in Section~\ref{sect:exam}. 

\section{Stable birational invariants of $G$-actions}
\label{sect:stable}
Let $X$ be a smooth projective variety and 
$G$ a finite group acting regularly on $X$.
It is known that 
$$
\rH^1(G,\Pic(X))
$$
is a stable $G$-birational invariant and that
$$
\rH^1(H, \Pic(X)) \neq 0, \quad H\subseteq G,
$$
is an obstruction to stable linearizability.
In this section we discuss several related invariants of $G$-actions.

Let 
$$
\Br(X):=\rH^2(X,\mathbb G_m), 
$$
be the Brauer group of $X$. We also consider the Brauer group 
of the quotient {\em stack}
$$
\Br([X/G])=\rH^2([X/G],\G_m).
$$
It is a stable $G$-birational invariant of $X$, as well as 
a stable birational invariant of the quotient stack $[X/G]$, in the sense of \cite[Sect.\ 4]{KT-orbi}.    

If $X$ is rational then
$$
\Br(X)=0. 
$$
A spectral sequence for $G$-actions yields:
\begin{align}
\begin{split}
\label{eqn.BrXmodG}
0&\to  \Hom(G,k^\times)\to \Pic(X,G)\to 
\Pic(X)^G \stackrel{\delta_2}{\lra}  \rH^2(G,k^\times)\\
&\qquad\to \Br([X/G])\to \rH^1(G,\Pic(X))\stackrel{\delta_3}{\lra}  \rH^3(G,k^\times).
\end{split}
\end{align}
We record the simple observations:
\begin{itemize}
\item Both $\delta_2$ and $\delta_3$ are zero, provided $G$ has a fixed point on $X$, e.g., when $G$ is cyclic.  
\item If $G$ is cyclic, then $\rH^2(G,k^\times)=0$.
\end{itemize}

The {\em Amitsur group} 
$$
\mathrm{Am}(X,G)
$$
is defined as the image of $\delta_2$ \cite[Sect.\ 6]{blanc2018finite}. It is a stable $G$-birational invariant. In particular,  
$$
\mathrm{Am}(X, H)\neq 0, \quad  H\subseteq G,
$$
is an obstruction to stable linearizability \cite{sarikyan}.

Analogously, the image of $\delta_3$ is a stable $G$-birational invariant. This follows from the corresponding property for $\rH^1(G,\Pic(X))$ and the functoriality of $\delta_3$. In \cite[Sect.\ 6]{KT-eff} we explained in the algebraic case how to compute $\delta_3$ effectively; the adaptation to the geometric case is straightforward. 

In Section~\ref{sect:exam} we give examples of $G$-actions on rational surfaces with nontrivial
$$
\Br([X/G]), \quad \rH^1(G,\Pic(X)), \quad \delta_3.
$$

\section{Computing the Brauer group}
\label{sect:brauer}
In this section, we assume that $X$ is a smooth projective rational surface, with a generically free regular action of a finite group $G$.
Our goal is to describe the invariant $\Br([X/G])$.
Our description makes use of:
\begin{itemize}
\item the injectivity of the restriction map $\Br([X/G])\to \Br([U/G])$,
for nonempty $G$-invariant open $U\subset X$ \cite[Prop.\ 2.5 (iv)]{ant}, hence also of restriction to the generic point
\begin{equation}
\label{eqn.togenericpoint}
\Br([X/G])\to \Br(k(X)^G);
\end{equation}
\item a description in terms of residues of the image of the map \eqref{eqn.togenericpoint} \cite[Prop.\ 2.2]{bssurf}.
\end{itemize}

The description in terms of residues is similar to the classical description of the Brauer group of a nonsingular projective variety as the subgroup of the Brauer group of the function field, having trivial residues along all divisors.
In the present setting, there are analogous residue maps, and
$\Br([X/G])$
is identified with 
\begin{equation}
\label{eqn.brauerkernel}
\ker\left(\Br(k(X)^G)\to  \bigoplus_{[\xi] \in X^{(1)}/G} \rH^1([\Spec(k(\xi))/D_{\xi}],\Q/\Z)\right).
\end{equation}
Here, $\xi$ denotes an orbit representative of codimension $1$ points on $X$, with decomposition group $$
D_\xi:=\{g\in G\,|\,\xi\cdot g=\xi\}.
$$
The decomposition group contains the inertia group $I_\xi$, consisting of the elements of $D_\xi$ which act trivially on $k(\xi)$.
The inertia group is cyclic, since
the normal bundle to the closure $\overline{\{\xi\}}$ at $\xi$ is a faithful one-dimensional representation, and central in $D_\xi$, since $k$ contains all roots of unity.

\begin{lemm}
\label{lem.itfactors}
The map in \eqref{eqn.brauerkernel} factors through
\[
\bigoplus_{[\xi] \in X^{(1)}/G} \rH^1(k(\xi)^{D_{\xi}},\Q/\Z).
\]
Fixing a Brauer class $\alpha\in \Br(k(X)^G)$ and an orbit representative $\xi$ of codimension $1$ points on $X$, the $\rH^1(k(\xi)^{D_{\xi}},\Q/\Z)$-component of the image of $\alpha$
is equal to $|I_\xi|$ times the image of $\alpha$ under the classical residue map
\[ \Br(k(X)^G)\to \rH^1(k(\xi)^{D_{\xi}},\Q/\Z). \]
\end{lemm}

\begin{proof}
This follows from \cite[Rem.\ 2.4]{bssurf}, which relates
the stacky residue map in \eqref{eqn.brauerkernel} to the classical residue map
in the statement of the lemma, and shows that the
former is $|I_\xi|$ times the latter.
\end{proof}

We recall the characterization of the Brauer group of the function field of a smooth projective rational surface $S$ over $k$, see \cite[Thm.\ 1]{AM}:
\begin{equation}
\label{eqn.am}
0\to \Br(k(S)) \to \bigoplus_{\text{curves }C\subset S} \,  \rH^1(k(C), \bQ/\bZ) \stackrel{r}{\lra} \bigoplus_{\text{$k$-points }\mathfrak p\in S}\, \bQ/\bZ. 
\end{equation}
Here the map $r$ is the sum over local ramification indices at points over $\mathfrak p$ of the {\em normalization} of $C$, for all curves $C$ passing through $\mathfrak p$.

\begin{prop}
\label{prop.brauerorbifold}
The Brauer group $\Br([X/G])$ is identified with the subgroup of $\Br(k(X)^G)$ of elements with $|I_\xi|$-torsion residue under the classical residue map to $\rH^1(k(\xi)^{D_{\xi}},\Q/\Z)$
for every orbit representative $\xi$ of codimension $1$ points of $X$.
Combining the residue map for every $\xi$, the resulting homomorphism
\[
\Br([X/G])
\to
\bigoplus_{[\xi] \in X^{(1)}/G} \rH^1(k(\xi)^{D_{\xi}},\Q/\Z).
\]
is injective.
\end{prop}

\begin{proof}
We have the isomorphism of $\Br([X/G])$ with the kernel \eqref{eqn.brauerkernel}.
By Lemma \ref{lem.itfactors}, this kernel is precisely the subgroup of $\Br(k(X)^G)$, satisfying the condition on residues in the statement of the proposition.
So, the first assertion is proved, and this gives the homomorphism in the second assertion.
It remains to establish injectivity.

Let $\alpha\in \Br(k(X)^G)$ be an element of the kernel.
We let $Y:=X/G$ be the quotient variety and take $S\to Y$ to be a resolution of singularities.
To show that $\alpha=0$, we are reduced by \eqref{eqn.am} to showing that the residue of $\alpha$ vanishes at all components of the exceptional divisor of $S\to Y$.
But the latter is a simple normal crossing divisor, with rational components, without loops.
So we get the triviality of the residues of $\alpha$ at
components of the exceptional divisor from the fact that the image of $\alpha$ in the middle term of \eqref{eqn.am} lies in the kernel of $r$.
\end{proof}

\begin{rema}
\label{rem.brauerorbifold}
The assertion about the injective homomorphism in Proposition \ref{prop.brauerorbifold} relies on the structure of the resolution of quotient surface singularities and may fail in higher dimension.
For instance, if $G$ is the Klein four-group, then the projectivization of the regular representation gives an action on $X\cong \bP^3$ with trivial inertia at all codimension $1$ points, but the exact sequence \eqref{eqn.BrXmodG} yields 
$$
0\ne \rH^2(G,k^\times)\cong \Br([X/G]).
$$
\end{rema}

The following result allows us to 
to determine the image of
$\rH^2(G,k^\times)$ in $\Br([X/G])$,
for given $G$,
in terms of the description 
of Proposition \ref{prop.brauerorbifold},
(see also Remark \ref{rem.mapfromH2G}).

\begin{prop}
\label{prop.mapfromH2G}
Let $\beta\in \rH^2(G,k^\times)$, with image in $\Br([X/G])$ restricting to $\alpha\in \Br(k(X)^G)$.
Let $\xi$ be an orbit representative of codimension $1$ points on $X$.
Then:
\begin{itemize}
\item[(i)] The residue of $\alpha$ in $\rH^1(k(\xi)^{D_{\xi}},\Q/\Z)$ is equal to the residue of the restriction of $\alpha$ to $\Br(k(X)^{D_\xi})$ in $\rH^1(k(\xi)^{D_{\xi}},\Q/\Z)$.
\item[(ii)] The restriction of $\alpha$ to $\Br(k(X)^{D_\xi})$ lies in the image of the inflation map
\[ \mathrm{inf}_{\xi} \colon \rH^2(D_\xi/I_\xi,(k(X)^{I_\xi})^\times)\to \rH^2(D_\xi,k(X)^\times), \]
of the Hochschild-Serre spectral sequence.
\item[(iii)] The residue of $\alpha$ in $\rH^1(k(\xi)^{D_{\xi}},\Q/\Z)$ is identified, under the isomorphism \eqref{eqn:h1-2}
\[ \rH^1(k(\xi)^{D_{\xi}},\Q/\Z)\cong \rH^2(k(\xi)^{D_{\xi}},\Z), \]
with the image under the valuation map
\[ \rH^2(D_\xi/I_\xi,(k(X)^{I_\xi})^\times)\to \rH^2(D_\xi/I_\xi,\Z) \]
of the lift of the restriction of $\alpha$ to $\Br(k(X)^{D_\xi})$ under the map $\mathrm{inf}_{\xi}$ in $\mathrm{(ii)}$.
The action of $D_\xi$ on $k(\xi)$ expresses $D_\xi/I_\xi$ as a quotient of the absolute Galois group of $k(\xi)^{D_\xi}$; the
corresponding inflation map sends $\rH^2(D_\xi/I_\xi,\Z)$ to $\rH^2(k(\xi)^{D_\xi},\Z)$.
\end{itemize}
\end{prop}

\begin{proof}
The map of quotient varieties $X/D_\xi\to X/G$ induces an isomorphism of residue fields at the points corresponding to $\xi$.
As well, it is \'etale in a neighborhood of $[\xi]$.
The residue map commutes with restriction by an \'etale morphism, so we have (i).

In the remainder of the proof, for notational simplicity, we suppose that $D_\xi=G$.
By a combination of the Hochschild-Serre spectral sequence and Hilbert's Theorem 90, the assertion in (ii) is equivalent to the triviality of the restriction of $\alpha$ to $\Br(k(X)^{I_\xi})$.
But $I_\xi$ is cyclic, so $\rH^2(I_\xi,k^\times)=0$, and the formation of $\alpha$ commutes with restriction.

For (iii), we only need to recall that a recipe to compute the residue of $\alpha$ is to find a finite Galois extension of $k(X)^G$, unramified over $[\xi]$, and a $2$-cocycle for the Galois group, whose class inflates to $\alpha$.
Applying the valuation to the $2$-cocycle, to get a $\Z$-valued $2$-cocycle, leads to a class which becomes the residue of $\alpha$, when we make the identification with a $\Q/\Z$-valued $1$-cocycle as in statement (iii). By (ii), we have this for the Galois extension $k(X)^{I_\xi}$.
\end{proof}

\begin{rema}
\label{rem.mapfromH2G}
The recipe in Proposition \ref{prop.mapfromH2G} (iii) comes down to obtaining a homomorphism $D_\xi/I_\xi\to \Q/\Z$.
This is determined by its restriction to cyclic subgroups, thus the computation is reduced to a treatment of the case $D_\xi/I_\xi$ is cyclic.
Then $D_\xi$ is a central extension of cyclic groups, so, bicyclic:
\[ D_\xi\cong \Z/m\Z\oplus \Z/n\Z, \]
with
$\rH^2(D_\xi,k^\times)$ cyclic of order $d:=\gcd(m,n)$; the inflation map
\[
\rH^2(\Z/d\Z\oplus\Z/d\Z,k^\times)\to \rH^2(D_\xi,k^\times)
\]
is an isomorphism.
This lets us reduce further to the case $$
m=n=d=|I_\xi|, \quad  I_\xi=\Z/d\Z\oplus 0.
$$
Let $\zeta$ denote a chosen primitive $d$th root of unity,
by which $\Z/d\Z\cong\mu_d$, and
let us write
$k(X)^{0\oplus \Z/d\Z}=k(X)^{D_\xi}(\gamma^{1/d})$,
where the chosen generator of $I_\xi$ acts on $\gamma^{1/d}$ by multiplication by $\zeta$.
A generator of 
$$
\rH^2(D_\xi,k^\times)\cong \rH^2(D_\xi,\mu_{\infty})
$$ 
is the cup product of the projections to the two factors.
The image in $\rH^2(D_\xi,k(X)^\times)$, in Proposition \ref{prop.mapfromH2G} (ii),
lifts to the class represented by $\gamma^{-1}\in k(X)^{D_\xi}$, under the standard isomorphism of $\rH^2(\Z/d\Z,(k(X)^{I_\xi})^\times)$ with a quotient by norms of $(k(X)^{D_\xi})^\times$.
\end{rema}

For the next result we impose the hypothesis that the $G$-action on $X$ is in \emph{standard form}, see \cite[Sect.\ 7.2]{HKTsmall}.
This means that there is a simple normal crossing divisor on whose complement $G$ acts freely, such that the $G$-orbit of every component of the divisor is smooth; general $X$ may be brought into standard form by a sequence of equivariant blow-ups.

\begin{coro}
\label{cor.puttogether}
Suppose that the $G$-action on $X$ is in standard form,
and let
$$
\alpha_\xi\in \rH^1(\Spec(k(\xi)^{D_{\xi}}),\Q/\Z),\qquad |I_\xi|\alpha_\xi=0,
$$
be given, for every orbit representative $\xi$ of codimension $1$ points of $X$.
There exists an $\alpha\in \Br([X/G])$, mapping under the residue map to
$(\alpha_\xi)_{[\xi]\in X^{(1)}/G}$, if and only if the local ramification indices at points over $\mathfrak p$ sum to $0$, for every orbit representative $\mathfrak p$ of $k$-points of $X$.
\end{coro}

\begin{proof}
The quotient variety $X/G$ has cyclic quotient singularities, over which a minimal resolution has chains of rational curves as exceptional divisors.
We conclude by the exact sequence \eqref{eqn.am}.
\end{proof}

\begin{rema}
Computations of Brauer groups of Deligne-Mumford stacks have been carried out in, e.g., \cite{ant}, however, mostly in the setting of points and curves with nontrivial generic stabilizers. 
\end{rema}

\section{Examples}
\label{sect:exam}
We give representative examples of computations of $\rH^1(G,\Pic(X))$.
In the absence of fixed points, when there is the possibility of subtle interplay with the other terms in \eqref{eqn.BrXmodG}, we also comment on these.

\subsection{Cyclic groups}
\label{sect:cyclic-h1}
Let $G$ be a cyclic group. 
In this case, there are always fixed points. The maps 
$\delta_2$ and $\delta_3$ are zero maps, and 
$$
\Br([X/G]) = \rH^1(G,\Pic(X)).
$$
Corollary \ref{cor.puttogether} allows us to compute the left side. Given the results in \cite{BP} and \cite{shinder}, we focus on cases when 
\begin{itemize}
\item The order of $G$ is composite,
\item some nontrivial stabilizers are proper subgroups of $G$. 
\end{itemize}

We assume that $X$ is {\em minimal}, i.e., $\Pic(X)^G\cong \bZ$ (del Pezzo case) or $\bZ^2$ (conic bundle case). 
We use the classification of nonlinear actions, summarized in \cite[Table 2]{blancsubgroups} and \cite[Chapter 8]{blanc-thesis}.

\begin{itemize}
\item[$(\mathbf{dJ})$]
We have
$$
G=C_{2n}=\langle \alpha\rangle \subset \mathrm{Cr}_2,
$$ 
and $\alpha^n$ is a de Jonqui\`eres involution, described in detail in \cite[Prop.\ 7.6.2, Prop.\ 7.6.3]{blanc-thesis}. 
The fixed locus of $\alpha^n$ contains a unique smooth curve, 
which is a hyperelliptic curve 
$$
X\supset C\stackrel{\pi}{\lra} \bP^1.
$$
The induced action of $\alpha$ on the base $\bP^1$ is cyclic of order $n$, with two fixed points, disjoint from the branch locus of $\pi$.
Thus the branch locus consists of $rn$ points, for some positive integer $r$, and $\mathsf{g}(C)=(rn-2)/2$.
There are three cases, depending on the number of fixed points for the action of $\alpha$ on $C$.

Applying Corollary~\ref{cor.puttogether} and the Riemann-Hurwitz formula for the genus $\mathsf{g}'=\mathsf{g}(C/G)$ of the quotient curve $C/G$ we obtain
$$
\rH^1(G, \Pic(X)) = \begin{cases}
(\bZ/2\bZ)^{r-2} & \text{4 fixed points, } \mathsf{g}'=(r-2)/2, \\
(\bZ/2\bZ)^{r-1} & \text{2 fixed points, } \mathsf{g}'=(r-1)/2,\\
(\bZ/2\bZ)^r & \text{no fixed points, } \mathsf{g}'=r/2.
\end{cases}
$$
\item[$(\mathbf{ndJ})$] We consider two actions, of $G=C_6$, described in \cite[Table 1, Table 2]{blancsubgroups} and in \cite{blanc-thesis}, using the notation of the latter to label the cases.  

\medskip
\noindent
Case {\tt 3.6.1}: We consider an action of $G=C_6=\langle \alpha\rangle$ on the cubic surface $X\subset \bP^3$, with equation
$$
w^3+x^3+y^3+xz^2+\lambda yz^2 =0, 
$$
with action by $\alpha$ on the coordinates $(w,x,y,z)$ by the weights
$$
[\zeta:1:1:-1], \quad  \zeta_3=e^{2\pi i /3}.
$$
There are
four points fixed by $\alpha$:
\[ \quad \,\, (0:1:-1:0),\, (0:1:-\omega:0),\, (0:1:-\omega^2:0),\, (0:0:0:1).  \]
We have a curve of genus $1$, given by $w=0$, fixed by $\alpha^2$, 
and another curve of genus $1$, given by $z=0$, fixed by $\alpha^3$.
Taking the quotients of these curves by residual actions we obtain rational curves with four, respectively three fixed points.
This allows us to conclude that 
$$
\Br([X/G])=0,
$$
since $2$-torsion and $3$-torsion elements summing to $0$ in $\bQ/\bZ$ must be $0$ individually in \eqref{eqn.am}.

\medskip
\noindent
Case {\tt 2.6}: We consider another action of $G=C_6$, on a del Pezzo surface of degree 2, $X\subset \bP(2,1,1,1)$, with equation
$$
w^2=x^3y+y^4+z^4+\lambda y^2z^2, 
$$
and with action on the coordinates $(w:x:y:z)$ by 
$$
[-1:\zeta_3:1:-1].
$$
The point $(0:1:0:0)$ is fixed by $\alpha$, there is a curve of genus $1$, given by  $x=0$, fixed by $\alpha^2$, and additional points
\[ (0:1:-1:0),\quad (0:1:-\omega:0),\quad (0:1:-\omega^2:0) \]
fixed by $\alpha^3$. The quotient of the genus 1 curve by residual action is again of genus $1$.
We conclude that 
$$
\Br([X/G])\cong (\Z/3\Z)^2.
$$
\end{itemize}

\subsection{Noncyclic actions, with fixed points}
In this case, the maps $\delta_2$ and $\delta_3$ are trivial, and we obtain the exact sequence
$$
0\to \rH^2(G,k^\times)\to \Br([X/G])\to \rH^1(G,\Pic(X))\to 0.
$$

\medskip
\noindent
Case {\tt 3.33.1}:  
Consider the action of $G=(\Z/3\Z)^2$ on the diagonal cubic surface 
$X\subset \bP^3$, with equation
\begin{equation}
\label{eqn:cub-dia}
w^3+x^3+y^3+z^3=0,
\end{equation}
where the generators of $G$ 
act via 
$$
g_1:=[\zeta:1:1:1] \quad \text{ and } \quad g_2:=[1:1:1:\zeta].
$$
In the number-theoretic setup, this action was realized as a Galois action, and \cite[Prop. 1]{ct-kanevsky} computed
$$
\rH^1(G,\Pic(X)) = \bZ/3.
$$
In our geometric approach, we find two elliptic curves $E_1,E_2$ fixed by the generators $g_1,g_2$, and intersecting in three points, given by $x^3+y^3=0$, and fixed by $G$. The quotients by the induced $\bZ/3$ on each curve are $\bP^1$. We find
$$
\Br([X/G])=(\bZ/3)^2.
$$
Taking into account that 
$$
\rH^2(G,k^\times) = \bZ/3,
$$
we conclude that
$$
\rH^1(G,\Pic(X)) = \bZ/3.
$$

\subsection{Noncyclic actions, without fixed points}
We turn to 

\medskip

\noindent
Case {\tt 3.333}: Consider the action of $G=(\Z/3\Z)^3$ on \eqref{eqn:cub-dia} via $g_1,g_2,$ and 
$$
g_3:=[1:1:\zeta:1].
$$
There are no fixed points. We know that $\delta_2$ is the zero map, since 
$$
\Pic(X,G)=\Pic(X)^G=\bZ (K_X),
$$
and that 
$$
\rH^2(G,k^\times) = (\bZ/3)^3, \quad \rH^3(G, k^\times) = 
(\bZ/3)^7.
$$
There are four elliptic curves
$E_1,E_2,E_3$, fixed by the generators, and $E_4$, fixed by the product of the generators $g_1,g_2,g_3$, respectively. Corollary \ref{cor.puttogether} gives 
$$
\Br([X/G])\cong (\Z/3)^3.
$$
We have $\delta_2$ trivial, $\rH^2(G,k^\times)$ mapping isomorphically to $\Br([X/G])$, and $\delta_3$ mapping 
$$
\rH^1(G,\Pic(X))\cong\Z/3,
$$
(known by \cite[Prop.\ 1]{ct-kanevsky}),
isomorphically to its image in $\rH^3(G,\bQ/\bZ)$; in particular, \emph{$\delta_3$ is nontrivial}.

\subsection{Dihedral group of order 8}
\label{sect:d8}
We consider an action of the dihedral group of order 8 on a minimal del Pezzo surface of degree 4, described in \cite[Sect.\ 6]{DI}. Let $X\subset \bP^4$ 
be given as an intersection of two quadrics
$$
x_0^2+x_1^2+x_2^2+x_3^2+x_4^2 = x_0^2+ax_1^2-x_2^2-ax_3^2=0, \quad a\neq -1,0,1.
$$
Let 
$$
G:= \langle \iota_0,\iota_2, \tau\rangle,
$$
where $\iota_i$ switches the sign of $x_i$, for $i=0$, $2$, and 
$$
\tau\colon (x_0:x_1:x_2:x_3:x_4) \mapsto (-x_2:-x_3:-x_0:-x_1:x_4).
$$
We have $\Pic(X)^G=\bZ$. 

We proceed to analyze curves in the fixed-point locus of non-identity elements of $G$. 
The elements $\iota_i$, $i=0,2$, fix the genus $1$ curves 
$$
E_i:=X\cap \{ x_i=0\},
$$
$E_0$ and $E_2$ meet in $4$ points, and
$\iota_i$ acts on $E_{2-i}$, fixing the $4$ points.
The map $\tau$ exchanges the genus $1$ curves. 

The conic
\[ C\,\,\colon\, x_0+x_2=x_1+x_3=2x_0^2+2x_1^2+x_4^2=0 \]
is fixed by $\tau$, with action by $\iota_0\iota_2$ fixing $2$ points.
The conic
\[ C'\,\,\colon\, x_0-x_2=x_1+x_3=2x_0^2+2x_1^2+x_4^2=0 \]
is similarly fixed by $\iota_0\iota_2\tau$. 

Two of the $4$ points are fixed by all of $G$, and the other two form a single $G$-orbit.
In particular, the maps $\delta_2$ and $\delta_3$ in \eqref{eqn.BrXmodG} are trivial.

We blow up the $4$ points to obtain $\widetilde{X}$, in standard form.
The exceptional divisor is in the fixed-point locus of $\iota_0\iota_2$.
Two of the components have a residual fixed-point free action of $\Z/2\Z\oplus \Z/2\Z$, the other two have a residual action of a group of order $2$.
For the computation of $\Br([X/G])$ we needed to pass to the model $\widetilde{X}$ to be able to apply
Corollary \ref{cor.puttogether}.
The result is:
\[ \Br([X/G])\cong (\Z/2\Z)^2. \]
Since $\rH^2(G,k^\times)= \Z/2\Z$, we obtain
\[ \rH^1(G,\Pic(X))=\Z/2\Z. \]

\section{Comparison with the equivariant Burnside group}
\label{scn.comparison}
In \cite{BnG} we have defined the \emph{equivariant Burnside group}
$$
\Burn_n(G),
$$
an abelian group capturing equivariant birational invariants, i.e., invariants of $G$-actions on algebraic varieties of dimension $n$ modulo equivariant birational maps. 
This group is generated by symbols
$$
(H,Z\actsfromleft K, \beta),
$$
where $H\subseteq G$ is an abelian group, 
$Z\subseteq Z_G(H)/H$, acting on a function field of transcendence degree $d\le n$ and 
$\beta$ is an unordered $(n-d)$-tuple of nontrivial characters of $H$, generating the character group. 
These are subject to certain relations; see \cite[Sect.\ 4]{BnG}.

The invariant of a faithful action of a finite group $G$ on a function field is computed on an appropriate model $X$ (where the $G$-action is in standard form), and is given by
$$
[X\actsfromright G]:=\sum \,\,
(H,Z\actsfromleft K, \beta) \in \Burn_n(G),
$$
a sum of contributions from loci with nontrivial stabilizers $H$, see \cite{BnG} and \cite{HKTsmall} for definitions and examples. In particular, an embedding $G\hookrightarrow \mathrm{Cr}_n$, up to conjugation in $\mathrm{Cr}_n$, gives rise to a well-defined class in $\Burn_n(G)$.

There is a subgroup
$$
\Burn_n^{\mathrm{inc}}(G) \subset \Burn_n(G)
$$
spanned by {\em incompressible} divisorial symbols  \cite[Defn. 3.3 and Prop. 3.4]{KT-vector}. 
This subgroup is a direct summand, and we can consider 
the projection
$$
\mathrm{inc}\colon \Burn_n(G)\to \Burn_n^{\mathrm{inc}}(G).
$$
Given an embedding 
$$
\iota\colon G\hookrightarrow \mathrm{Cr}_n,
$$
we obtain a sum of incompressible divisorial symbols
\[
\mathrm{inc}([\iota])\in \Burn_n^{\mathrm{inc}}(G).
\]

Concretely, when $n=2$, the incompressible divisorial symbols are those where $K$ is the function field of a curve of positive genus or $Z$ is noncyclic, acting on $K\cong k(t)$ \cite[Prop. 3.6]{KT-vector}.

A related invariant, for actions of elements $g$ of finite order
of the plane Cremona group, was introduced in \cite{deFer} and refined in \cite{blancsubgroups}:
\begin{itemize}
\item {\em Normalized fixed curve}: $\mathrm{NFC}(g)$ is the normalization of the component of positive genus of the fixed curve of $g$, or $\emptyset$ if no such component exists.
(It is known that there can be at most one such component.)
\item {\em Normalized fixed curve with action}:
say $g\in G$ has order $m$, then
$$
\mathrm{NFCA}(g):=\left((\mathrm{NFC}(g^r ), g|_{\mathrm{NFC}(g^r)})  \right)^{m-1}_{r=1},
$$
where the second factor records the residual automorphism on the fixed curve.
\end{itemize}
This invariant distinguishes {\em cyclic} actions, in the following sense:

\begin{theo}
\cite{blancsubgroups}
\label{thm:blanc}
Two cyclic subgroups $G$ and $G'$ of $\mathrm{Cr}_2$
of the same order are conjugate if and only if
$$ 
\mathrm{NFCA}(g) = \mathrm{NFCA}(g'), 
$$
for some generators $g$ of $G$ and $g'$ of $G'$.
\end{theo}

We now explain the relation between this invariant and the Burnside group formalism, 
for cyclic groups $G$.  
We fix a generator $g$ of $G$, and consider an embedding $\iota\colon G\to \mathrm{Cr}_2$.
If the term $(H,Z\actsfromleft K,\beta)$ arises
in $\mathrm{inc}([\iota])$,
then $Z=G/H$, $H=\langle g^r\rangle$ for some $r$, and the $r$th component of $\mathrm{NFCA}(g)$ consists of a curve of positive genus with function field $K$, with an action of $\bar g\in G/H$.
It is clear that $\mathrm{NFCA}(g)$ is completely determined by $\mathrm{inc}([\iota])$.

However, except when $r$ is half the order of $G$, the $\beta$-component of
$(H,Z\actsfromleft K,\beta)$ carries extra information, which is not captured by $\mathrm{NFCA}(g)$.
This sheds some light on the formulation of Theorem \ref{thm:blanc}, which is stated in terms of conjugacy of subgroups rather than conjugacy of chosen generators in the Cremona group.
Even with the $\beta$-components, however, Theorem \ref{thm:blanc} cannot be
directly strengthened
to a characterization of
conjugacy classes of \emph{embeddings} $G\hookrightarrow \mathrm{Cr}_2$, as the next example reveals.

\begin{exam}
\cite[Thm. 4]{blancsubgroups}
\label{exa.iota}
Let $G=\langle g\rangle$ be a finite cyclic group of order $4$,
and let $S$ be the degree $1$ del Pezzo surface
defined by
\[ w^2=z^3+z(ax^4+bx^2y^2+cy^4)+xy(a'x^4+b'x^2y^2+c'y^4) \]
in $\bP(3,1,1,2)$.
For general coefficients $a$, $b$, $c$, $a'$, $b'$, $c'$,
the embeddings
\[ \iota, \iota'\colon G\to \Aut(S)\subset \mathrm{Cr}_2, \]
where $g$ acts
by scalar multiplication on the coordinates $w$, $x$, $y$, $z$ by \[ [i:1:-1:-1],\quad \text{respectively}\quad [-i:1:-1:-1], \]
are not conjugate.
But
$$
\mathrm{inc}([\iota]) = \mathrm{inc}([\iota'])
$$
in $\Burn^{\mathrm{inc}}_2(G)$. 
\end{exam}

On the other hand, Theorem~\ref{thm:blanc} does imply that 
$\mathrm{inc}([\iota])$  determines 
$\rH^1(G,\Pic(X))$,
for cyclic groups $G$.

\bibliographystyle{plain}
\bibliography{burnsimple}

\end{document}